\documentclass[11pt]{amsart}
\theoremstyle{definition}
\usepackage{indentfirst, amssymb,amsmath,amsthm}
\usepackage{csquotes}
\usepackage{hyperref}
\usepackage{mathrsfs}
\newtheorem*{theoA}{Theorem A}
\newtheorem{theorem}{Theorem}[section]
\newtheorem{lem}{Lemma}[section]
\newtheorem{defi}{Definition}[section]
\newtheorem{rem}{Remark}[section]

\begin{document}
\title[An inequality regarding differential polynomial]{An inequality regarding differential polynomial}
\date{}
\author[S. Saha]{Sudip Saha$^{1}$}
\date{}
\address{$^{1}$Department of Mathematics, Ramakrishna Mission Vivekananda Centenary College, Rahara,
West Bengal 700 118, India.}
\email{sudipsaha814@gmail.com}
\maketitle
\let\thefootnote\relax
\begin{abstract}
In this paper, we prove an inequality regarding the differential polynomial. This improves some recent results.
\end{abstract}
\vspace{1cm}
\section{Introduction and Main Results}
Throughout this paper, we assume that the reader is familiar with the value distribution theory (\cite{Hy}). Further, it will be convenient to let that $E$ denote any set of positive real numbers of finite Lebesgue measure, not necessarily same at each occurrence. For any non-constant meromorphic function $f$, we denote by $S(r,f)$ any quantity satisfying $$S(r, f) = o(T(r, f))~~\text{as}~~r\to\infty,~r\not\in E.$$
~~~~Let $f$ be a non-constant meromorphic function. A meromorphic function $a(z)(\not\equiv 0,\infty)$ is called a \enquote{small function} with respect to $f$ if $T(r,a(z))=S(r,f)$. For example, polynomial functions are small functions with respect to any transcendental entire function.
\begin{defi} (\cite{f})
Let $a\in \mathbb{C}\cup\{\infty\}$.  For a positive integer  $k$, we denote
\begin{enumerate}
\item [i)] by $N_{k)}\left(r,a;f\right)$ the counting function of $a$-points of $f$ whose multiplicities are not greater than $k$,
\item [ii)] by $N_{(k}\left(r,a;f\right)$ the counting function of $a$-points of $f$ whose multiplicities are not less than $k$.
\end{enumerate}
Similarly, the reduced counting functions $\overline{N}_{k)}(r,a;f)$ and $\overline{N}_{(k}(r,a;f)$ are defined.
\end{defi}
\begin{defi}(\cite{ld})
For a positive integer $k$, we denote $N_{k}(r,0;f)$ the counting function of zeros of $f$, where a zero of $f$ with multiplicity $q$ is counted $q$ times if $q\leq k$, and is counted $k$ times if  $q> k$.
\end{defi}
In 2003, I. Lahiri and S. Dewan proved the following theorem:
\begin{theoA}(\cite{ld})
Let $f$ be a transcendental meromorphic function and $\alpha (\not \equiv 0, \infty)$ be a small function of $f$. If $\psi=\alpha(f)^{n}(f^{(k)})^p$, where $n(\geq 0), p(\geq 1), k(\geq 1)$ are integers, then for any small function $a(\not \equiv 0, \infty)$ of $\psi$ we have
$$ (p+n)T(r,f) \leq \overline{N}(r,\infty;f)+\overline{N}(r,0;f)+pN_{k}(r,0;f)+\overline{N}(r,a;\psi)+S(r,f).$$ 
\end{theoA}
In this paper we extend and improve the Theorem A. To state our next result, we recall a well known definition.
\begin{defi}(\cite{chak})
Let $n_{0j}, n_{1j}, \cdots, n_{kj}$ be non-negative integers. Then the expression $M_{j}[f]=(f)^{n_{0j}}(f')^{n_{1j}}\cdots(f^{(k)})^{n_{kj}}$ is called a differential monomial generated by $f$. The quantities $d(M_{j})= \sum \limits_{i=0}^{k} n_{ij}$ and $ \Gamma_{M_{j}}= \sum \limits_{i=0}^{k} (i+1)n_{ij}$ are known as the degree and weight of the monomial $M_{j}$ respectively. \par 
The sum $P[f]= \sum \limits_{j=1}^{t} b_{j}M_{j}[f]$ is called a differential polynomial generated by $f$, where $T(r,b_{j})=S(r,f)~(j=1,2, \cdots, t).$ The quantities $\overline{d}(P) = \max \limits_{1\leq j \leq t} \{d(M_{j})\} $ and $\Gamma_{P} = \max \limits_{1\leq j \leq t} \{\Gamma_{M_{j}}\} $ are called degree and weight of the polynomial $P[f]$ respectively. \par 
The numbers $\underline{d}(P) = \min \limits_{1\leq j \leq t} \{d(M_{j})\} $ and $k$ (the highest order of the derivative of $ f$ in $P[f]$) are known as the lower degree and order of the polynomial $P[f]$.\par
$P[f]$ is called homogeneous if $\underline{d}(P)=\overline{d}(P)$. $P[f]$ is called a linear differential polynomial if $\overline{d}(P)=1.$ Otherwise $P[f]$ is called a non-linear differential polynomial. 
\end{defi}
\vspace{0.1cm}
\begin{theorem}\label{th1.1}
Let $f(z)$ be a transcendental meromorphic function and $\alpha (z)(\not \equiv 0, \infty)$ be a small function of $f(z)$. Let $P[f] = \alpha \sum \limits_{j=1}^{t} M_j [f]$ be a differential polynomial generated by $f$, where $M_j[f] = c_j(f)^{n_{0j}}(f')^{n_{1j}} \cdots (f^{(k)})^{n_{kj}}~(j=1,2,\cdots,t)$ such that $k(\geq 1)$ is the order of $P[f]$,$~t(\geq 1),~n_{ij}(i=0,1,\cdots,k;j=1,2,\cdots ,t)$ are non-negative integers and $c_j(j=1,2,\cdots,t)$ are small functions of $f$ such that they do not have poles at the zeros of $f$. Then, for a small function $a(\not \equiv 0, \infty)$, 
\begin{eqnarray*}
\underline{d}(P) T(r,f) &\leq & \overline{N}(r,0;f)+\overline{N}(r,a;P[f])+\overline{N}(r,\infty;f) +n_{1e}N_{1}(r,0;f) \\
&~& + n_{2e}N_{2}(r,0;f)+\cdots +n_{ke}N_{k}(r,0;f) +S(r,f), 
\end{eqnarray*}
where $ (1\cdot n_{1e} + 2\cdot n_{2e}+ \cdots +k \cdot n_{ke})=\max \limits_{1\leq j \leq t}(1\cdot n_{1j} + 2\cdot n_{2j}+ \cdots +k \cdot n_{kj})$.
\end{theorem}
\begin{rem}
Clearly, Theorem \ref{th1.1} extends and improves  Theorem A.
\end{rem}
\section{Necessary Lemmas}
\begin{lem}\label{lem1}(\cite{ham})
Let $A > 1$, then there exists a set $M(A)$ of upper logarithmic density at most
$ \delta(A) = \min \{(2e^{(A-1)}-1)^{-1}, 1+e(A-1)\exp(e(1-A))\}$ such that for $k = 1, 2, 3,\cdots$ 
$$\limsup \limits_{r \to \infty, r \notin M(A)} \frac{T(r,f)}{T(r,f^{(k)})} \leq 3eA.$$
\end{lem}
\begin{lem}\label{lem2} Let $f$ be a transcendental meromorphic function and  $\alpha~(\not \equiv 0, \infty)$ be a small function of $f$. Let, $\psi = \alpha(f)^{q_0}(f')^{q_1} \cdots (f^{(k)})^{q_k}$, where $q_0, q_1, \cdots, q_k(\geq 1), k(\geq 1)$ are non-negative integers. Then $\psi$ is not identically constant.
\end{lem}
\begin{proof}
Since, $\alpha$ is a small function of $f$, then $T(r,\alpha)=S(r,f)$. Therefore the proof follows from Lemma(3.4) of (\cite{cspk}).
\end{proof}
\begin{lem}\label{lem3}
 Let $f$ be a transcendental meromorphic function and  $\alpha~(\not \equiv 0, \infty)$ be a small function of $f$. Let, $\psi = \alpha(f)^{q_0}(f')^{q_1} \cdots (f^{(k)})^{q_k}$, where $q_0, q_1, \cdots, q_k(\geq 1), k(\geq 1)$ are non-negative integers. Then $$ T(r,\psi) \leq \left\{q_0+2q_1+\cdots+(k+1)q_k\right\}T(r,f)+S(r,f).$$
\end{lem}
\begin{lem}[\cite{balone, chuchi}]\label{lem4}
Let, $f$ be a meromorphic function and $P[f]$ be a differential polynomial. Then $$ m\left(r,\frac{P[f]}{f^{\overline{d}(P)}}\right) \leq (\overline{d}(P)-\underline{d}(P))m\left(r,\frac{1}{f}\right)+S(r,f).$$
\end{lem}
\section{Proof of the Theorem}
\begin{proof}[\textbf{Proof of Theorem \ref{th1.1}}]
Since $P[f]$ is a differential polynomial generated by $f$, then using Lemma \ref{lem4}
\begin{eqnarray}
\nonumber T(r,(f)^{\overline{d}(P)}) &=& N(r,0;(f)^{\overline{d}(P)}) + m\left(r,\frac{1}{(f)^{\overline{d}(P)}}\right)+O(1) \\
\nonumber &=& N(r,0;(f)^{\overline{d}(P)}) + m\left(r,\frac{P[f]}{(f)^{\overline{d}(P)}}\frac{1}{P[f]}\right)+O(1)\\
\nonumber &\leq & N(r,0;(f)^{\overline{d}(P)}) + (\overline{d}(P)-\underline{d}(P))m\left(r,\frac{1}{f}\right)+ m\left(r,\frac{1}{P[f]}\right) +S(r,f)\\
\nonumber &= & N(r,0;(f)^{\overline{d}(P)}) + (\overline{d}(P)-\underline{d}(P))m\left(r,\frac{1}{f}\right)+ ~T(r,P[f]) \\
\nonumber &~~ &  - N(r,0;P[f]) +S(r,f).
\end{eqnarray}
\begin{eqnarray}
\therefore \label{eq7} T(r,(f)^{\overline{d}(P)})&\leq&  N(r,0;(f)^{\overline{d}(P)}) + (\overline{d}(P)-\underline{d}(P))m\left(r,\frac{1}{f}\right)+ ~T(r,P[f]) \\
\nonumber &~~ &  - N(r,0;P[f]) +S(r,f).
\end{eqnarray}
Using Nevanlinna's second fundamental theorem, from (\ref{eq7}) we get 
\begin{eqnarray}\label{eq8}
T(r,(f)^{\overline{d}(P)}) &\leq &  N(r,0;(f)^{\overline{d}(P)}) + \overline{N}(r,0;P[f]) + \overline{N}(r,\infty;P[f])\\ 
\nonumber & &+ \overline{N}(r,a;P[f])- N(r,0;P[f])+ (\overline{d}(P)-\underline{d}(P))m\left(r,\frac{1}{f}\right) \\
\nonumber & &+S(r,P[f]) +S(r,f).
\end{eqnarray}
From definition of $P[f]$ and using ~Lemma (\ref{lem3}) we have $ T(r,P[f]) \leq K_1T(r,f) +S(r,f),$ for some constant $K_1$. This implies $ S(r,P[f])=S(r,f).$ We also note that $ \overline{N}(r,\infty;P[f]) = \overline{N}(r,\infty;f)+S(r,f).$ \\
Thus from (\ref{eq8}), 
\begin{eqnarray}\label{eq9}
T(r,(f)^{\overline{d}(P)}) &\leq &  N(r,0;(f)^{\overline{d}(P)}) + \overline{N}(r,0;P[f]) + \overline{N}(r,\infty;f)\\ 
\nonumber &~&+ \overline{N}(r,a;P[f])- N(r,0;P[f]) + (\overline{d}(P)-\underline{d}(P))m\left(r,\frac{1}{f}\right)\\
\nonumber & & +S(r,f).
\end{eqnarray}\\
\par
Let, $z_0$ be a zero of $f(z)$ with multiplicity $q~(\geq 1)$. For any $k$, $M_{j}[f]~~(j=1,2, \cdots, t)$ has a zero at $z_0$ of order atleast
\begin{eqnarray*}
&~&qn_{0j} + (q-1)n_{1j} + (q-2)n_{2j}+ \cdots + 2n_{q-2~j} +n_{q-1~j}+r_j\\
&= & q(n_{0j} + n_{1j} + \cdots + n_{q-1~j}) -(1\cdot n_{1j} + 2\cdot n_{2j}+ \cdots +(q-1) \cdot n_{q-1~j}) +r_j\\
&= & q(n_{0j} + n_{1j} + \cdots + n_{q-1~j}+\cdots +n_{kj})-q(n_{qj}+n_{q+1~j}+\cdots+n_{kj}) \\
&~~& -(1\cdot n_{1j} + 2\cdot n_{2j}+ \cdots +(q-1) \cdot n_{q-1~j}) +r_j\\
&=&q(d(M_j))-(1\cdot n_{1j} + 2\cdot n_{2j}+ \cdots +(q-1) \cdot n_{q-1~j}+ qn_{qj}+ \cdots+ qn_{kj})+r_j \\
&\geq & q(\underline{d}(P)) -(1\cdot n_{1j} + 2\cdot n_{2j}+ \cdots +(q-1) \cdot n_{q-1~j}+ qn_{qj}+ \cdots+ qn_{kj}), ~~~~\text{if}~~~~q \leq k
\end{eqnarray*}
and 
\begin{eqnarray*}
&~&qn_{0j} + (q-1)n_{1j} + (q-2)n_{2j}+ \cdots + (q-k)n_{kj}+r_j\\
&= & q(n_{0j} + n_{1j} + \cdots + n_{kj}) -(1\cdot n_{1j} + 2\cdot n_{2j}+ \cdots +k \cdot n_{kj})+r_j\\
&= & q(d(M_j)) - (1\cdot n_{1j} + 2\cdot n_{2j}+ \cdots +k \cdot n_{kj})+r_j\\
&\geq & q(\underline{d}(P)) - (1\cdot n_{1j} + 2\cdot n_{2j}+ \cdots +k \cdot n_{kj}), ~~~~\text{if}~~~~q > k,
\end{eqnarray*}
where $r_j~(j=1,2, \cdots ,t)$ is the multiplicity of zero of $c_j~(j=1,2, \cdots ,t)$ at $z_0$, which must be a non-negative quantity. 
\par
Therefore, $P[f]$ has a zero at $z_0$ of order
\begin{eqnarray*}
&\geq& q(\underline{d}(P))+r -(1\cdot n_{1e}+\cdots +(q-1) \cdot n_{q-1~e}+ qn_{qe}+ \cdots+q \cdot n_{ke}) ~~~\text{if}~~~q \leq k ~~\text{and} \\
&\geq& q(\underline{d}(P))+r -(1\cdot n_{1e} + 2\cdot n_{2e}+ \cdots +k \cdot n_{ke})~~~\text{if}~~~q > k,
\end{eqnarray*} 
where $r=0$ if $\alpha (z)$ does not have a zero or pole at $z_0$, $r=s$ if $\alpha (z)$ has a zero of order $s$ at $z_0$, $r=-s$ if $\alpha (z)$ has a pole of order $s$ at $z_0$, $s$ being a natural number and $\{ n_{1e},n_{2e}, \cdots,n_{ke}\}$ is the set of values such that $ (1\cdot n_{1e} + 2\cdot n_{2e}+ \cdots +k \cdot n_{ke})=\max \limits_{1\leq j \leq t}(1\cdot n_{1j} + 2\cdot n_{2j}+ \cdots +k \cdot n_{kj}) $. [We see that whenever $q(\geq 1)$ and $k(\geq 1)$ are fixed, then $ \max \limits_{1\leq j \leq t}(1\cdot n_{1j} + 2\cdot n_{2j}+ \cdots +q \cdot n_{kj})$ and $\max \limits_{1\leq j \leq t}(1\cdot n_{1j} + 2\cdot n_{2j}+ \cdots +k \cdot n_{kj})$ will appear for same set of values $\{ n_{1j},n_{2j}, \cdots,n_{kj}\}$ for some $j \in \{1,2, \cdots,t\}$.]\\

Therefore, 
\begin{eqnarray}\label{eq10}
&~& 1 + q(\overline{d}(P)) -q(\underline{d}(P))-r \\
\nonumber &~~~~~& + (1\cdot n_{1e}+\cdots +(q-1) \cdot n_{q-1~e}+ qn_{qe}+ \cdots+q \cdot n_{ke})\\
\nonumber &= &  q(\overline{d}(P)-\underline{d}(P))+1 -r \\
\nonumber &~~~& + (1\cdot n_{1e}+\cdots +(q-1) \cdot n_{q-1~e}+ qn_{qe}+ \cdots+q \cdot n_{ke}) ~~~~\text{if}~~~~q \leq k   
\end{eqnarray} 
and 
\begin{eqnarray}\label{eq11}
&~& 1 + q(\overline{d}(P)) -q(\underline{d}(P))-r + (1\cdot n_{1e} + 2\cdot n_{2e}+ \cdots +k \cdot n_{ke})\\
\nonumber &= & q(\overline{d}(P)-\underline{d}(P))+1 -r + (1\cdot n_{1e} + 2\cdot n_{2e}+ \cdots +k \cdot n_{ke}) ~~~~\text{if}~~~~q > k.   
\end{eqnarray} 
Therefore, from (\ref{eq10}) and (\ref{eq11}) we have
\begin{eqnarray*}
N(r,0;(f)^{\overline{d}(P)})+\overline{N}(r,0;P[f])-N(r,0;P[f]) &\leq & (\overline{d}(P)-\underline{d}(P))N(r,0;f)\\
& &+\overline{N}(r,0;f)+ n_{1e}N_{1}(r,0;f)\\ 
&~~& + n_{2e}N_{2}(r,0;f)+\cdots \\
& &+n_{ke}N_{k}(r,0;f) +S(r,f). 
\end{eqnarray*}
Therefore (\ref{eq9}) gives
\begin{eqnarray*}
\overline{d}(P) T(r,f) &\leq & \overline{N}(r,0;f)+\overline{N}(r,a;P[f])+\overline{N}(r,\infty;f)+(\overline{d}(P)-\underline{d}(P))N(r,0;f) \\
& &+ (\overline{d}(P)-\underline{d}(P))m\left(r,\frac{1}{f}\right)+n_{1e}N_{1}(r,0;f) + n_{2e}N_{2}(r,0;f)+\cdots \\
&~&+n_{ke}N_{k}(r,0;f) +S(r,f)  \\
&=& \overline{N}(r,0;f)+\overline{N}(r,a;P[f])+\overline{N}(r,\infty;f)+(\overline{d}(P)-\underline{d}(P))T(r,f) \\
& & +n_{1e}N_{1}(r,0;f) + n_{2e}N_{2}(r,0;f)+\cdots +n_{ke}N_{k}(r,0;f) +S(r,f).
\end{eqnarray*}
\begin{eqnarray*}
\therefore \underline{d}(P) T(r,f) &\leq & \overline{N}(r,0;f)+\overline{N}(r,a;P[f])+\overline{N}(r,\infty;f)+n_{1e}N_{1}(r,0;f)\\
& & + n_{2e}N_{2}(r,0;f)+\cdots +n_{ke}N_{k}(r,0;f) +S(r,f).
\end{eqnarray*}
\end{proof}
\begin{center}
{\bf Acknowledgement}
\end{center}
The author is thankful to his supervisor Dr. Bikash Chakraborty for guiding him throughout the preparation of this paper. The author is also grateful to the anonymous referee for his/her valuable suggestions which considerably improved the presentation of the paper.\par
The author is thankful to the Council of Scientific and Industrial Research, HRDG, India for granting Junior Research
Fellowship (File No.: 08/525(0003)/2019-EMR-I) during the tenure of which this work was done.


\begin{thebibliography}{99}
\bibitem{chak} S. S. Bhoosnurmath, B. Chakraborty, and H. M. Srivastava, A note on the value distribution of differential polynomials, Commun. Korean Math. Soc., 34 (2019), No. 4, 1145-1155.
\bibitem{balone} B. Chakraborty, A simple proof of the Chuang's inequality, Analele Universitatii de Vest, Timisoara, Seria Matematica-Informatica, LV, 2, (2017), 85-89.
\bibitem{cspk} B. Chakraborty, S. Saha, A .K. Pal and J. Kamila, Value distribution of some differential monomials, Filomat, Accepted for publication.
\bibitem{chuchi} Chuang, Chi-tai, On differential polynomials, Analysis of one complex variable (Laramie, Wyo., 1985), 12–32, World Sci. Publishing, Singapore, 1987.
\bibitem{hn} W. K. Hayman, Picard values of meromorphic functions and their derivatives, Ann. Math., 70(1959), 9-42.
\bibitem{ham} W. K. Hayman and J. Miles, On the growth of a meromorphic function and its derivatives,
Complex Variables, 12 (1989), 245–260.
\bibitem{Hy} W. K. Hayman, Meromorphic Functions, The Clarendon Press, Oxford (1964).
\bibitem{hg} X. Huang and Y. Gu, On the value distribution of $f^{2}f^{(k)}$, J. Aust. Math. Soc., 78(2005), 17-26.
\bibitem{KS} H. Karmakar and P. Sahoo, On the Value Distribution of $f^{n}f^{(k)}-1$, Results Math., 73(2018), doi:10.1007/s00025-018-0859-9.
\bibitem{ld} I. Lahiri and S. Dewan, Inequalities arising out of the value distribution of a differential monomial, J. Inequal. Pure Appl. Math. 4(2003), no. 2, Art. 27.
\bibitem{f} C. C. Yang and H. X. Yi, Uniqueness Theory of Meromorphic Functions, Kluwer Academic Publishers, Dordrecht, The Netherlands, (2003).
\end{thebibliography}
\end{document}